\documentclass{amsart}
\usepackage{amssymb}

\usepackage{hyperref}
\usepackage{xcolor}

\allowdisplaybreaks

\newtheorem{theorem}{Theorem}[section]

\newtheorem{corollary}[theorem]{Corollary}
\newtheorem{proposition}[theorem]{Proposition}

\theoremstyle{definition}

\newtheorem{example}[theorem]{Example}

\theoremstyle{remark}
\newtheorem{remark}[theorem]{Remark}

\numberwithin{equation}{section}

\newcommand{\C}{\mathbb{C}}
\newcommand{\D}{\mathbb{D}}
\newcommand{\lD}{\lambda_{\D}}
\newcommand{\pD}{\partial_{\D}}
\newcommand{\p}{\partial}
\newcommand{\pzb}{\p\overline{z}}
\newcommand{\Om}{\Omega}
\newcommand{\la}{\lambda}
\newcommand{\Oml}{\Omega_\lambda}
\newcommand{\ov}{\overline}
\newcommand{\dt}{\mathbb{T}}

\begin{document}

\title[Hyperbolic curvature of holomorphic level curves]{Hyperbolic curvature of holomorphic level curves}

\author[Mihai Iancu]{Mihai Iancu}
\author[Veronica-Oana Nechita]{Veronica-Oana Nechita}
\address{Faculty of Mathematics and Computer Science,
Babe\c{s}-Bolyai University, \newline\indent  1 M. Kog\u{a}lniceanu Str., 400084
Cluj-Napoca, Romania} 
\email{mihai.iancu@ubbcluj.ro}
\email{veronica.nechita@ubbcluj.ro}


\subjclass[2020]{30F45, 30C45, 30C35, 52A10}

\date{\today}

\keywords{hyperbolic convexity, hyperbolic curvature, hyperbolic area, hyperbolic perimeter, level curve, radius of convexity.}

\begin{abstract}
We give sharp bounds for the hyperbolic curvature of the level curve $|z|=|f(z)|$, when $f:\D\to\D$ is holomorphic on the unit disc $\D$ and $f(0)\neq0$, as well as for other related level curves. As a consequence, we point out a rigidity theorem: if the hyperbolic curvature of the above level curve vanishes at some point, then the level curve is a hyperbolic geodesic and $f$ is an automorphism. As another  consequence, we prove that $\frac{1}{\sqrt 2}$ is the greatest lower bound of the supremum of $r\in(0,1)$ such that the level curve $|z|=r|f(z)|$ is (Euclidean) convex. This constant turns out to be also the radius of convexity for hyperbolically convex self-maps of $\D$ that fix the origin. We also give (sharp) estimates for the total hyperbolic curvature, hyperbolic area and hyperbolic perimeter of the sublevel sets.
\end{abstract}

\maketitle

\section{Introduction}
\label{sec1}

Let $\mathbb{D}$ be the unit disc in the complex plane $\mathbb{C}$.
 Solynin \cite{Sol} answered positively a question raised by Mej\'ia and Pommerenke \cite{MejPom} regarding the hyperbolic convexity (i.e., convexity with respect to hyperbolic geodesic segments) of the set
$$\Omega(f):=\{z\in\D:|z|<|f(z)|\},$$
when $f:\D\to\D$ is holomorphic with $f(0)\neq0$.
Recently, Efraimidis and Gumenyuk \cite{EfrGum} provided an alternative proof, by proving the hyperbolic convexity of the more general sets studied by Arango, Mej\'ia, Pommerenke \cite{AraMejPom}:
$$\Omega_\la(f):=\left\{z\in\D:\frac{1-|f(z)|^2}{1-|z|^2}<\la\right\},$$
 when $f:\D\to\D$ is holomorphic and $\la\ge1$, with $f(0)\neq0$ if $\la=1$. The study of $\Om(f)$ has interesting applications in Probability Theory (see \cite{JenPom, MejPom, MejPom2}) and in One-Parameter Semigroups Theory (see \cite{ElShSu}).

In this paper, we want to ``measure'' the hyperbolic convexity of these sets. Since the main instrument to do this is the hyperbolic curvature, we provide some bounds for the hyperbolic curvature of the level curves given by the  boundaries. Using some of the found estimates, we derive some  rigidity results for $\Omega_\la(f)$ and we find the radius of convexity $r\in(0,1)$ for $\Omega(rf)$. Moreover, we provide some estimates for the total hyperbolic curvature, hyperbolic area and hyperbolic perimeter of these sublevel sets.

Let's start by recalling the definition of the hyperbolic curvature in $\D$. Let the coefficients of the first fundamental form of an oriented regular surface $S$ be $E(u,v)=G(u,v)=\frac{1}{(1-u^2-v^2)^2}, F(u,v)=0$, for $(u,v)\in \mathbb{D}$, with respect to a local parametrization ${\rm x}: \mathbb{D}\to S$. Every $p\in {\rm x}(\D)$ is a hyperbolic point of $S$, since the Gaussian curvature $K$ of $S$ at $p$ is $-4$ (see, e.g., \cite[Section 4-3, Exercise 2]{doC}). $\lD:\D\to(0,\infty)$, $\lD(z)=\frac{1}{1-|z|^2}$, is the \textit{hyperbolic density} of $\D$. Let $\gamma$ be a regular (i.e., with everywhere nonvanishing tangent) $C^2$ curve in $\D$.
Let $k_h(z,\gamma)$ be the {\it  hyperbolic curvature} of $\gamma$ at $z\in\{\gamma\}$ which is given by (see \cite{FerGra, FlOsg, J, MaMin, Osg})
\begin{align}
\label{kh}
\begin{split}
    k_h(z,\gamma)&=\frac{1}{\lD(z)}\left(k_e(z,\gamma)-\frac{\partial \log \lD}{\partial n}(z)\right)\\
    &=(1-|z|^2)k_e(z,\gamma)-2{\rm Re}(n\ov{z}),
\end{split}
\end{align}
where $n=i\cdot\gamma'/|\gamma'|$ and $k_e(z,\gamma)$ is the signed Euclidean curvature of $\gamma$ at $z=\gamma(t)$. Note that $k_h(z,\gamma)$ is independent of the parametrization, it depends only on the range $\{\gamma\}$ and its orientation. Since $k_e(\gamma(t),\gamma) = \frac{{\rm Im}(\gamma''(t)\overline{\gamma'(t)})}{|\gamma'(t)|^3}$,
\begin{equation}
\label{kht}
    k_h(\gamma(t),\gamma)=\frac{1}{|\gamma'(t)|}{\rm Im}\left((1-|\gamma(t)|^2)\frac{\gamma''(t)}{\gamma'(t)}+2\overline{\gamma(t)}\gamma'(t)\right).
\end{equation}
Taking into account the expression of the {\it geodesic curvature} of ${\rm x}\circ\gamma$ in terms of the coefficients of the first fundamental form (see \cite[Section 4-4]{doC} and \cite{FlOsg}), $k_h(z,\gamma)$ coincides with the geodesic curvature of ${\rm x}\circ \gamma$ at ${\rm x}(z)$. Therefore, for every $\varphi\in{\rm Aut}(\D)$, since $\varphi$ is an isometry with respect to $\lD$, 
\begin{equation}
\label{inv}
    k_h(z,\gamma) = k_h(\varphi(z),\varphi\circ\gamma), 
\end{equation}
for all $z\in\{\gamma\}$ (see \cite{FlOsg, J, Osg}).

For every $\zeta\in\D$, let $\varphi_\zeta\in{\rm Aut}(\D)$ be given by
$$\varphi_\zeta(z)=\frac{\zeta-z}{1-\overline{\zeta}z},\quad z\in\D.$$ In view of \eqref{kh} and \eqref{inv}, for every $z\in\{\gamma\}$,
\begin{equation}
\label{khke}
k_h(z,\gamma)=k_e(0,\varphi_z\circ\gamma).    
\end{equation}

There is a connection between the geodesic curvature, Gaussian curvature and surface area, given by the (Local) Gauss-Bonnet Theorem (see \cite[Section 4-5]{doC}):  since $K\equiv-4$ on ${\rm x}(\D)$, if $\gamma$ is a regular Jordan $C^2$ curve in $\D$,  
\begin{equation}
    \label{GB}
    k_h(\gamma)-4 A_h(D_\gamma)=2\pi,
\end{equation}
where $D_\gamma$ is the interior of $\{\gamma\}$,  
$$A_h(D_\gamma)=\iint_{D_\gamma} \sqrt{EG-F^2}{\rm d}u{\rm d}v=\iint_{D_\gamma}\la_\D^2(z){\rm d}A(z)$$ is the {\it hyperbolic area} of $D_\gamma$ (with respect to the Lebesgue measure $A$ in $\D$) and 
$$k_h(\gamma)=\int_\gamma k_h(z,\gamma)\la_\D(z)|{\rm d}z|$$ 
is the {\it total hyperbolic curvature} of $\gamma$. Let us also mention the  \textit{hyperbolic isoperimetric inequality} 
\begin{equation}
    \label{isop}
    L_h^2(\gamma)\ge 4\pi A_h(D_\gamma)+4A_h^2(D_\gamma),
\end{equation}
with equality if and only if $\gamma$ is a (hyperbolic) circle, where 
$$L_h(\gamma)=\int_{\gamma} \la_\D(z)|{\rm d}z|$$ is the {\it hyperbolic perimeter} of $D_\gamma$ or length of $\gamma$ (see \cite{Oss}). Note that $k_h$, $A_h$ and $L_h$ are all invariant with respect to the automorphisms of $\D$.

We point out the following rigidity theorem. The proof will be given in Section \ref{sec3}, after we extract, in Section \ref{sec2}, a lower bound for the hyperbolic curvature from the proof of Efraimidis and Gumenyuk for \cite[Theorem 3.1]{EfrGum}.

\begin{theorem}
\label{t01}
    Let $f:\D\to\D$ be a holomorphic function such that $f(0)\neq0$.
    Then $k_h(\zeta,\pD\Omega(f))=0$ for some $\zeta\in\pD\Omega(f)$ if and only if $f\in{\rm Aut}(\D)$.
    Moreover, if one of the conditions holds, then
    $\pD\Omega(f)$ is a hyperbolic geodesic.
\end{theorem}

A few remarks are in order. By \cite[Proposition 3.3]{EfrGum}, $\Omega(f)\neq\D$. Moreover, by the Schwarz Lemma, $f(0)\neq0 \Leftrightarrow \Omega(f)\neq\emptyset$. So, $\pD\Omega(f)\neq\emptyset$, where $\pD$ is the boundary in $\D$. By \cite[Theorem 2.2]{MejPom}, $\pD\Omega(f)$ is either an analytic starlike (in particular, Jordan) curve or a union of open analytic arcs with endpoints on $\partial\D$, so, for every $\zeta\in\pD\Omega(f)$, by $k_h(\zeta,\pD\Omega(f))=0$ we mean $k_h(\zeta,\gamma)=0$, where $\gamma$ is a regular parametrization of an arc of $\pD\Omega(f)$ containing $\zeta$. Note that $k_h(\zeta,\pD\Omega(f))$ is uniquely determined up to sign, depending on the orientation of the parametrization. As a convention, we choose a parametrization $\gamma$ of the arc of $\pD\Omega(f)$ containing $\zeta$ such that the unit normal $n$ is inward-pointing. 

In Section \ref{sec4}, we give several lower bounds for the hyperbolic, respectively Euclidean, curvature of the level curve $\frac{\lambda_\D(z)}{\lambda_\D(f(z))}=\lambda$, for $f:\D\to\D$ holomorphic and $\la\ge1$, with $f(0)\neq0$ if $\la=1$, as consequences of the estimate given in Section \ref{sec2}. We consider for every lower bound the equality case. Let's mention such a result. 

\begin{theorem}
\label{t02}
         Let $f:\D\to\D$ be holomorphic and $\lambda>1$. Then, for every $\zeta\in\pD\Omega_\lambda(f)$,
     \begin{equation*}
         k_h(\zeta,\pD\Omega_\lambda(f))\ge \frac{\la(\la-1)(1-|\zeta|^2)}{|\overline{f'(\zeta)}f(\zeta)-\lambda\zeta|}.
     \end{equation*}
    Moreover, equality holds for some $\zeta\in\pD\Oml(f)$ if and only if $f\in{\rm Aut}(\D)$ with $\la<\frac{1+|f(0)|}{1-|f(0)|}$. Furthermore, if equality holds for some $\zeta\in\pD\Oml(f)$, then it holds for every $\zeta\in\pD\Oml(f)$.
\end{theorem}
Let $\la>1$. Then $0\in\Om_\la(f)$, and thus, by \cite[Proposition 3.3]{EfrGum}, $\pD\Om_\la(f)\neq\emptyset$ if and only if $\sup_{\D}|f'|>\la$ or $f$ is not a finite Blaschke product. Moreover, if $f\in{\rm Aut}(\D)$ and $\la\ge \frac{1+|f(0)|}{1-|f(0)|}$, then $\pD\Om_\la(f)=\emptyset$, by \cite[Example 3.2]{EfrGum}. In view of  \cite{AraMejPom}, $\Oml(f)$ is a starlike domain and $\pD\Om_\la(f)$, if non-empty, is a smooth arc around each of its points. Again, as a convention, we choose a parametrization
of an arc of $\pD\Oml(f)$ such that the unit normal $n$ is inward-pointing. 

In Section \ref{sec5}, we consider the special case of the Jordan level curves $\p\Om(rf)$, when $r\in(0,1)$, and, using a sharp inequality obtained by Ma and Minda \cite{MaMin} for hyperbolically convex functions in $\D$ and the fixed point function studied by Mej\'ia and Pommerenke \cite{MejPom}, we get the following sharp estimates. 

\begin{theorem}
\label{t03}
    Let $f:\D\to\ov{\D}$ be holomorphic such that $f(0)\neq0$. Then the following sharp inequalities hold, for all $r\in(0,1)$ and $\zeta\in\p\Om(rf)$,
    \begin{equation*}
    \frac{1+r}{1-r}C_{rf,\zeta}-\frac{2(1-r^2)}{r}\frac{1}{C_{rf,\zeta}} \ge k_h(\zeta,\p\Om(rf))\ge\frac{1-r}{1+r}C_{rf,\zeta}+\frac{2(1-r^2)}{r}\frac{1}{C_{rf,\zeta}},
    \end{equation*}
    where $C_{rf,\zeta}=|\ov{(rf)'(\zeta)}(rf)(\zeta)-\zeta|\frac{1-|\zeta|^2}{|\zeta|^2}$.
\end{theorem}

As an application of Theorem \ref{t03} and its proof, in Section \ref{sec6}, we find the greatest lower bound of the supremum of $r\in(0,1)$ such that $\Om(rf)$ is (Euclidean) convex.

\begin{theorem}
\label{t04}
    Let  $\mathcal{F}=\{f:\D\to\ov{\D}: f \text{ is holomorphic with }f(0)\neq0\}$. Then 
    $$\inf_{f\in\mathcal{F}}\sup\{r\in(0,1):\Om(rf)\mbox{ is convex}\}=\frac{1}{\sqrt{2}}.$$
\end{theorem}

The proof of Theorem \ref{t04} also gives the radius of (Euclidean) convexity for the family of univalent self-maps of $\D$ that fix the origin and have hyperbolically convex image.

In view of \eqref{GB} and \eqref{isop}, it is natural to take a look at the estimates of the total hyperbolic curvature, hyperbolic area and hyperbolic perimeter of the corresponding sublevel sets. In Section \ref{sec7}, we point out some estimates, taking into account the results of Kourou \cite{Kou1, Kou2}. Let's mention such a result.

\begin{theorem}
\label{t05}
    Let $f:\D\to\ov{\D}$ be holomorphic such that $f(0)\neq0$ and $r\in(0,1)$. Then 
    $$\frac{\pi(1-r)r^2|f(0)|^2}{(1+r)((1+r)^2-4r|f(0)|^2)}\le A_h(\Om(rf))\le \frac{\pi r^2|f(0)|^2}{1-r^2}.$$
    Moreover, for each inequality, the equality holds if and only if  $f\equiv\sigma\in\dt$.
\end{theorem}

This result combined with \eqref{GB} provides sharp bounds for the total hyperbolic curvature $k_h(\p\Om(rf))$. We also give sharp bounds for $L_h(\p\Om(rf))$, sharp hyperbolic isoperimetric-type inequalities for $\p\Om(rf)$, and some bounds for $A_h(\Om_\la(f))$ and $L_h(\pD\Om_\la(f))$, which, however, are not sharp. 

\section{A lower bound for the hyperbolic curvature of level curves}
\label{sec2}

Following \cite{EfrGum} (see also \cite{AraMejPom}), we consider
$$\Omega_\lambda(f)=\{z\in\D:u(z):=|f(z)|^2-\lambda|z|^2+\lambda-1>0\},$$
for $f:\D\to\D$ holomorphic and $\lambda>0$.

\begin{remark}
$$\Omega_\lambda(f)=\left\{z\in\D:\frac{\lambda_\D(z)}{\lambda_\D(f(z))}=\frac{1-|f(z)|^2}{1-|z|^2}<\lambda\right\}.$$
    Also, $\Omega_1(f)=\Omega(f)=\{z\in\D:|z|<|f(z)|\}$.

\end{remark}

\begin{remark}
\label{r_grad}
    Let $\lambda\ge1$ and $f:\D\to\D$ be holomorphic, with $f(0)\neq0$ if $\la=1$. By the first part of the proof of \cite[Theorem 3.1]{EfrGum}, 
    \begin{equation}
    \label{grad}
        \nabla u(\zeta)=2\frac{\partial u}{\partial \overline{z}}(\zeta)=2(\overline{f'(\zeta)}f(\zeta)-\lambda\zeta)\neq0,
    \end{equation}
    for every $\zeta\in\pD\Omega_\lambda(f)$, if $\lambda>1$ or if $\lambda=1$ and $f\notin{\rm Aut}(\D)$. However, if $\lambda=1$ and $f\in{\rm Aut}(\D)$, then $f=e^{i\theta}\varphi_a$, for some $a\in\D\setminus\{0\}$ and $\theta\in\mathbb{R}$, and thus $\nabla u(\zeta)=0$, for some $\zeta\in\pD\Oml(f)$, if and only if $a(1-\ov{a}\zeta)=\ov{a}\zeta(a-\zeta)$, which is in contradiction with $|f(\zeta)|=|\zeta|<1$. So, \eqref{grad} holds, and thus $\pD\Oml(f)$ is a smooth arc around each $z\in\pD\Oml(f)$, for every $f:\D\to\D$ holomorphic with $f(0)\neq0$, when $\lambda=1$, and for every $f:\D\to\D$ holomorphic, when $\la>1$ (cf. \cite[Section 3]{AraMejPom},\cite[Section 2]{MejPom}).
\end{remark}

In the following, we are going to reveal the expression of the hyperbolic curvature in the proof of Efraimidis and Gumenyuk \cite[Theorem 3.1]{EfrGum}. 

\begin{theorem}
\label{t0}
     Let $\lambda\ge1$ and $f:\D\to\D$ be holomorphic, with $f(0)\neq0$ if $\la=1$. Then, for every $\zeta\in\pD\Omega_\lambda(f)$,
    $$k_h(\zeta,\pD\Omega_\lambda(f))\ge \frac{\lambda^2\left(1+|\zeta|^2-2|f(\zeta)|\right)-|f'(\zeta)|^2(1-|f(\zeta)|)^2}{\la|\overline{f'(\zeta)}f(\zeta)-\lambda\zeta|}.$$
    Moreover, if equality holds for some $\zeta\in\pD\Oml(f)$, then $f$ is a Blaschke product of degree at most $2$.
\end{theorem}

\begin{proof}
Let $\zeta\in\pD\Omega_\lambda(f)$ and $\gamma$ be a local regular parametrization of $\pD\Omega_\lambda(f)$ around $\gamma(0)=\zeta$ such that $\gamma'(0)=-i\frac{\partial u}{\partial \overline{z}}(\zeta)/\left|\frac{\partial u}{\partial \overline{z}}(\zeta)\right|$, using, in view of Remark \ref{r_grad}, $\frac{\partial u}{\partial \overline{z}}(\zeta)\neq0$. Moreover, the unit normal vector $n=\nabla u(\zeta)/|\nabla u(\zeta)|=i\gamma'(0)$ of $\gamma$ at $\zeta$ is inward-pointing with respect to $\Oml(f)$. In view of \eqref{kht},
\begin{align*}
k_h(\zeta,\pD\Oml(f))&=k_h(\zeta,\gamma)  \\
    & = \frac{1-|\zeta|^2}{|\frac{\p u}{\pzb}(\zeta)|}{\rm Re}\left(\gamma''(0)\frac{\p u}{\p z}(\zeta)-\frac{2\ov{\zeta}}{1-|\zeta|^2}\frac{\p u}{\pzb}(\zeta)\right).
\end{align*}

Since $u\circ\gamma=0$, $(u\circ\gamma)''(0)=0$ and thus 
$${\rm Re}\left(\frac{\p^2u}{\p z^2}(\zeta)(\gamma'(0))^2\right)+\frac{\p^2u}{\p z\pzb}(\zeta)|\gamma'(0)|^2+{\rm Re}\left(\frac{\p u}{\p z}(\zeta)\gamma''(0)\right)=0.$$
So,
\begin{align*}
k_h(\zeta,\gamma)  =& -\frac{1-|\zeta|^2}{|\frac{\p u}{\pzb}(\zeta)|}\left({\rm Re}\left(\frac{\p^2u}{\p z^2}(\zeta)(\gamma'(0))^2\right)\right.\\
&\left.+\frac{\p^2u}{\p z\pzb}(\zeta)+\frac{2}{1-|\zeta|^2}{\rm Re}\left(\ov{\zeta}\frac{\p u}{\pzb}(\zeta)\right)\right).
\end{align*}

Since,
$$\varphi_{-\zeta}'(0)=|\zeta|^2-1\quad\mbox{ and }\quad \varphi_{-\zeta}''(0)=-2\ov{\zeta}(|\zeta|^2-1)$$
and $\frac{\p u}{\p z}(\zeta)\big(\gamma'(0)\big)^2=-\frac{\p u}{\pzb}(\zeta)$, we have
\begin{align*}
 k_h(\zeta,\gamma) & = -\frac{1}{(1-|\zeta|^2)|\frac{\p u}{\pzb}(\zeta)|}\left({\rm Re}\left(\frac{\p^2u}{\p z^2}(\zeta)(-\varphi_{-\zeta}'(0)\gamma'(0))^2\right)\right.\\
 &\left.+\frac{\p^2u}{\p z\pzb}(\zeta)|-\varphi_{-\zeta}'(0)\gamma'(0)|^2+{\rm Re}\left(\frac{\p u}{\p z}(\zeta)(-\varphi_{-\zeta}''(0))(\gamma'(0))^2\right)\right)\\
 & =  -\frac{1}{(1-|\zeta|^2)|\frac{\p u}{\pzb}(\zeta)|} \frac{v''(0)}{2},
\end{align*}
where $v(t)=u(-\varphi_{-\zeta}(tk))$ and $k=\gamma'(0)$.

In view of \eqref{inv}, for $\alpha\in\mathbb{R}$,
$$k_h(\zeta,\gamma)=k_h(\zeta_\alpha,\gamma_\alpha),$$
where $\zeta_\alpha=e^{-i\alpha}\zeta$ and $\gamma_\alpha=e^{-i\alpha}\gamma$ is a parametrization around $\zeta_\alpha$ of $\pD\Oml(f_\alpha)$ $=e^{-i\alpha}\pD\Oml(f)$, $f_\alpha(z)=f(e^{i\alpha}z)$. Moreover, for $\alpha,\beta\in\mathbb{R}$, since $\Oml(f_\alpha)=\Oml(f_{\alpha,\beta})$, where $f_{\alpha,\beta}=e^{i\beta}f_{\alpha}$, we can choose $\alpha$ and $\beta$ such that $\zeta_\alpha=|\zeta|$ and $f_{\alpha,\beta}(\zeta_\alpha)=|f(\zeta)|$.
Let  $u_{\alpha,\beta}(z)=|f_{\alpha,\beta}(z)|^2-\la|z|^2+\la-1$, $z\in\D$. Note that $f_{\alpha,\beta}'(\zeta_\alpha)=e^{i(\alpha+\beta)}f'(\zeta)$ and $\frac{\p u_{\alpha,\beta}}{\pzb}(\zeta_\alpha)=e^{-i\alpha}\frac{\p u}{\pzb}(\zeta)$. In view of the proof of \cite[Theorem 3.1]{EfrGum},
$$\frac{v_{\alpha,\beta}''(0)}{2}\le \la(1-\zeta_\alpha^2)\left[(1-\zeta_\alpha^2)(\la|c_0|^2-1)+2b^2|c_0|^2-2\zeta_\alpha^2+2b|c_1|\right],$$
where $v_{\alpha,\beta}(t)=u_{\alpha,\beta}(-\varphi_{-\zeta_\alpha}(tk_\alpha))$, $k_\alpha=e^{-i\alpha}k$, $b=f_{\alpha,\beta}(\zeta_\alpha)$ and $c_0=h'(0)$, $c_1=h''(0)/2$, $h=\varphi_b\circ f_{\alpha,\beta}\circ(-\varphi_{-\zeta_\alpha})$. By the Schwarz Lemma applied to $h$, $|c_0|\le 1$. By the Schwarz-Pick Lemma applied to $z\mapsto h(z)/z$, $|c_1|\le 1-|c_0|^2$ and, if equality holds, then either $f\in{\rm Aut}(\D)$, in the case $|c_0|=1$, or $f$ is a Blaschke product of degree 2, in the case $|c_0|<1$. Also, we observe that $c_0=-\frac{f_{\alpha,\beta}'(\zeta_\alpha)}{\la}$.

Taking into account the above and using again $(1-|\zeta|^2)\la=1-|f(\zeta)|^2$, we get
$$k_h(\zeta,\gamma)\ge \frac{\lambda(1-|\zeta|^2)\left[1+|\zeta|^2-2|f(\zeta)|-\frac{|f'(\zeta)|^2}{\la^2}(1-|f(\zeta)|)^2\right]}{(1-|\zeta|^2)|\frac{\p u}{\pzb}(\zeta)|}.$$ So, the desired inequality holds.
\end{proof}

\begin{remark}
\label{Dhj}
    The coefficients $c_0$ and $c_1$ in the proof of Theorem \ref{t0} satisfy
    $$|c_0|=|D_{h1}f(\zeta)|\quad\text{ and }\quad 2|c_1|=|D_{h2}f(\zeta)|,$$
    where, in view of \cite[Definition 4]{MaMin2}, $D_{h1},D_{h2}$ are given, for every $z\in\D$, by
    $$D_{h1}f(z)=\frac{(1-|z|^2)f'(z)}{1-|f(z)|^2},$$
    $$D_{h2}f(z)=\frac{(1-|z|^2)^2f''(z)}{1-|f(z)|^2}+\frac{2(1-|z|^2)^2\ov{f(z)}f'(z)^2}{(1-|f(z)|^2)^2}-\frac{2\ov{z}(1-|z|^2)f'(z)}{1-|f(z)|^2},$$
    or by $D_{h1}f(z)=\tilde{f}_z'(0),D_{h1}f(z)=\tilde{f}_z''(0)$, where $\tilde{f}_z=-\varphi_{f(z)}\circ f \circ(-\varphi_{-z})$. $D_{h1}$ is also known as the {\it hyperbolic derivative}. 
\end{remark}

\section{A rigidity theorem involving the hyperbolic curvature}
\label{sec3}

\begin{remark}
\label{r1}
    If $\gamma$ is an arc of circle in $\D$ with endpoints on $\p\D$, oriented clockwise, then $\gamma$ has constant hyperbolic curvature equal to $2\cos\theta$, where $\theta$ is the angle between the $\gamma$ and $\p\D$ (anticlockwise oriented); see \cite[Example 1]{MejMin}; cf. \cite[Section 2.4]{FerGra}. This can be easily seen by applying $\varphi_\zeta$ to $\gamma$, and then using \eqref{khke} and the fact that the Euclidean curvature of  a circle oriented anticlockwise is the inverse of the radius, for every $\zeta\in\{\gamma\}$. In particular, $\gamma$ is a hyperbolic geodesic if and only if $k_h(\cdot,\gamma)\equiv 0$.
\end{remark}

\begin{example}
    \label{ex1}
    Let $f=e^{i\alpha}\varphi_a\in{\rm Aut}(\D)$, where $\alpha\in\mathbb{R}$, $a\in\D\setminus\{0\}$, and let $\la>0$. Then (see \cite[Example 1]{AraMejPom} and \cite[Example 3.2]{EfrGum}) 
    $$\pD\Oml(f)=\left\{\zeta\in\D:\left|\zeta-b\right|=\sqrt{\frac{|b|^2-1}{\la}}\right\},$$
    where $b=1/\overline{a}$, and $\pD\Oml(f)\neq\emptyset\Leftrightarrow\frac{1-|a|}{1+|a|}<\la<\frac{1+|a|}{1-|a|}$. By Remark \ref{r1},
    $$k_h(\zeta,\pD\Oml(f))=\sqrt{\frac{|b|^2-1}{\la}}\cdot(\la-1),$$
    for every $\zeta\in\pD\Oml(f)$. In particular, if $\la=1$, then $\pD\Om(f)$ is a hyperbolic geodesic.
\end{example}

A domain $\Omega\subset\D$ is said to be \textit{hyperbolically convex}, if, for every $z_1,z_2\in\Omega$, the hyperbolic geodesic segment connecting $z_1$ and $z_2$ is in $\Omega$.  $f:\D\to\D$ is said to be hyperbolically convex (see \cite{MaMin}), if $f$ is a conformal map onto a hyperbolically convex domain.

By \cite{Sol} (see also \cite{EfrGum}), if $f\notin{\rm Aut}(\D)$ and $f(0)\neq0$,  then $\Omega(f)$ is \textit{strictly} hyperbolically convex, i.e., for every $\zeta\in\pD\Omega(f)$ there exists a hyperbolic geodesic $\gamma_{\zeta}$, called {\it supporting} hyperbolic geodesic, passing through $\zeta$ such that $\{\gamma_{\zeta}\}\cap\pD\Omega(f)=\{\zeta\}$ (by the proof of \cite[Lemma 0]{MejMin}, it is sufficient to verify this property locally). In particular, $k_h(\zeta,\pD\Omega(f))\ge0$, for all $\zeta\in\pD\Omega(f)$ (see \cite[Theorem 1]{MaMin}). Note that, while the strict positiveness of the hyperbolic curvature of the boundary implies the strict hyperbolic convexity of the interior (see the proof of  \cite[Proposition 1]{MejMin}), the converse is not necessarily true, as the following example shows.

\begin{example}
    For $\varepsilon>0$, let $\gamma_\varepsilon(t)= t+it^4,t\in[-\varepsilon,\varepsilon]$. Let $\gamma_1,\gamma_2$ be arcs of circles such that each arc has an endpoint on $\partial\D$ and the other endpoint satisfies $\gamma_1(-\varepsilon)=\gamma_\varepsilon(-\varepsilon)$, $\gamma_1'(-\varepsilon)=\gamma_\varepsilon'(-\varepsilon)$, $\gamma_1''(-\varepsilon)=\gamma_\varepsilon''(-\varepsilon)$, respectively $\gamma_2(\varepsilon)=\gamma_\varepsilon(\varepsilon)$, $\gamma_2'(\varepsilon)=\gamma_\varepsilon'(\varepsilon)$, $\gamma_2''(\varepsilon)=\gamma_\varepsilon''(\varepsilon)$. By \eqref{kht}, $k_h(\gamma_\varepsilon(t),\gamma)=12t^2+O(t^4)$, as $t\to0$. So, we can choose a sufficiently small $\varepsilon>0$ such that the domain $\Omega\subset\D$ whose boundary in $\D$ is given by the $C^2$ regular arc $\Gamma=\gamma_1\cup\gamma_\varepsilon\cup\gamma_2$ is strictly hyperbolically convex. Indeed, $k_h(\zeta,\Gamma)>0$, for all $\zeta\in\pD\Omega\setminus\{0\}$, so we have a supporting hyperbolic geodesic at $\zeta$, and, even though $k_h(0,\Gamma)=0$, we have a supporting hyperbolic geodesic at $0$ given by the diameter $(-1,1)$.
\end{example}

The following rigidity result, which restates Theorem \ref{t01}, is a consequence of Theorem \ref{t0}.

\begin{corollary}
\label{t1}
    Let $f:\D\to\D$ be a holomorphic function such that $f(0)\neq0$.
    Then $k_h(\zeta,\pD\Omega(f))=0$ for some $\zeta\in\pD\Omega(f)$ if and only if $f\in{\rm Aut}(\D)$. Moreover, if one of the conditions holds, then
    $\pD\Omega(f)$ is a hyperbolic geodesic.
\end{corollary}

\begin{proof}
Taking $\lambda=1$ in Theorem \ref{t0} and using the fact that $\zeta\in\pD\Omega(f) \Leftrightarrow |\zeta|=|f(\zeta)|$, we have
for every $\zeta\in\pD\Omega(f)$,
\begin{equation}
    \label{khlb}
  k_h(\zeta,\pD\Omega(f))\ge \frac{(1-|\zeta|)^2(1-|f'(\zeta)|^2)}{|\overline{f'(\zeta)}f(\zeta)-\zeta|}.  
\end{equation}

Now, the necessary condition follows from \eqref{khlb} and the Schwarz-Pick Lemma (using also $|\zeta|=|f(\zeta)|$): $|f'(\zeta)|\le1$; $|f'(\zeta)|=1\Leftrightarrow f\in{\rm Aut}(\D)$. The sufficient condition holds, because, if $f\in{\rm Aut}(\D)$, then $\pD\Om(f)$ is a hyperbolic geodesic in $\D$, by Example \ref{ex1}.
\end{proof}

\section{Lower bounds for curvatures of level curves}
\label{sec4}

In this section, we give more consequences of Theorem \ref{t0}. For each lower bound for hyperbolic, respectively Euclidean, curvature, we consider the equality case.

In view of \cite[Theorem 3.1]{EfrGum}, if $\la>1$, then $\Om_\la(f)$ is strictly hyperbolically convex, for every $f:\D\to\D$ holomorphic. We now restate and prove Theorem \ref{t02}. 

\begin{corollary}
\label{ct1}
     Let $f:\D\to\D$ be holomorphic and $\lambda>1$. Then, for every $\zeta\in\pD\Omega_\lambda(f)$,
     \begin{equation}
     \label{khlb3}
         k_h(\zeta,\pD\Omega_\lambda(f))\ge \frac{\la(\la-1)(1-|\zeta|^2)}{|\overline{f'(\zeta)}f(\zeta)-\lambda\zeta|}.
     \end{equation}
    Moreover, equality holds for some $\zeta\in\pD\Oml(f)$ if and only if $f\in{\rm Aut}(\D)$ with $\la<\frac{1+|f(0)|}{1-|f(0)|}$. Furthermore, if equality holds for some $\zeta\in\pD\Oml(f)$, then it holds for every $\zeta\in\pD\Oml(f)$.
\end{corollary}

\begin{proof}
   In Theorem \ref{t0}, we apply the Schwarz-Pick inequality and we use the fact that $\zeta\in\pD\Omega_\lambda(f)\Leftrightarrow \frac{1-|f(\zeta)|^2}{1-|\zeta|^2}=\la$, to get
      \begin{align*}
       k_h(\zeta,\pD\Omega_\lambda(f))\ge & \frac{\lambda^2\left(1+|\zeta|^2-2|f(\zeta)|\right)-|f'(\zeta)|^2(1-|f(\zeta)|)^2}{\la|\overline{f'(\zeta)}f(\zeta)-\lambda\zeta|} \\
       \ge & \frac{\lambda\left[1+|\zeta|^2-2|f(\zeta)|-(1-|f(\zeta)|)^2\right]}{|\overline{f'(\zeta)}f(\zeta)-\lambda\zeta|} \\
    =& \frac{\la(|\zeta|^2-|f(\zeta)|^2)}{|\overline{f'(\zeta)}f(\zeta)-\lambda\zeta|}\\
    =& \frac{\la(1-|f(\zeta)|^2-(1-|\zeta|^2))}{|\overline{f'(\zeta)}f(\zeta)-\lambda\zeta|}\\
    =& \frac{\la(\la-1)(1-|\zeta|^2)}{|\overline{f'(\zeta)}f(\zeta)-\lambda\zeta|}.
   \end{align*}
   
Next, assume that there exists $\zeta\in\pD\Oml(f)$ such that equality in \eqref{khlb3} holds. Then, by the Schwarz-Pick Lemma, $f\in{\rm Aut}(\D)$. In view of Example \ref{ex1}, $\la<\frac{1+|f(0)|}{1-|f(0)|}$. Conversely, if $f=e^{i\alpha}\varphi_a\in{\rm Aut}(\D)$, with $\alpha\in\mathbb{R}$, $a\in\D$ and $\la<\frac{1+|a|}{1-|a|}$, then, by Example \ref{ex1} again, $\pD\Oml(f)\neq\emptyset$ and one can easily verify that, for every $\zeta\in\pD\Oml(f)$ $\left(\Leftrightarrow |1-\ov{a}\zeta|^2=\frac{1-|a|^2}{\la}\right)$, we have
    $k_h(\zeta,\pD\Oml(f))=\sqrt{\frac{1-|a|^2}{\la}}\frac{\la-1}{|a|}=\frac{\la(\la-1)(1-|\zeta|^2)}{|\overline{f'(\zeta)}f(\zeta)-\lambda\zeta|}$.
\end{proof}

\begin{corollary}
\label{cr3}
    Let  $f:\D\to\D$ be a holomorphic function such that $f(0)\neq0$. Then, for every $\zeta\in\pD\Om(f)$, 
    \begin{equation}
    \label{khlb2}
        k_h(\zeta,\pD\Omega(f))\ge \frac{(1-|\zeta|)^2}{2|\zeta|}(1-|f'(\zeta)|^2).
    \end{equation}
    Moreover, equality holds for some $\zeta\in\pD\Om(f)$ if and only if $f\in{\rm Aut}(\D)$. Furthermore, if equality holds for some $\zeta\in\pD\Om(f)$, then it holds for every $\zeta\in\pD\Om(f)$, with the lower bound being identically 0.
\end{corollary}

\begin{proof}
The inequality follows from \eqref{khlb}, using, in the denominator, the Schwarz-Pick inequality, the triangle inequality and the fact that $\zeta\in\pD\Om(f)$ $\Leftrightarrow|f(\zeta)|=|\zeta|$. By the Schwarz-Pick Lemma, the equality holds for some $\zeta\in\pD\Om(f)$ if and only if  $f\in{\rm Aut}(\D)$. As in the proof of Corollary \ref{t1}, if $f\in{\rm Aut}(\D)$, then $\pD\Om(f)$ is a hyperbolic geodesic.
\end{proof}

\begin{corollary}
\label{cr2}
 Let $f:\D\to\D$ be holomorphic and $\lambda>1$. Then, for every $\zeta\in\pD\Omega_\lambda(f)$,
    \begin{equation}
    \label{kh3}
         k_h(\zeta,\pD\Omega_\lambda(f))\ge |\zeta|-|f(\zeta)|.
     \end{equation}
Moreover, equality holds for some $\zeta\in\pD\Oml(f)$ if and only if $f\in{\rm Aut}(\D)$ with $\la\le\frac{1}{1-|f(0)|^2}$. Furthermore, equality holds for at most one point.
\end{corollary}

\begin{proof}
Applying the Schwarz-Pick inequality and the triangle inequality,
we have $|\overline{f'(\zeta)}f(\zeta)-\lambda\zeta|\le \la(|f(\zeta)|+|\zeta|)$. So, by Corollary \ref{ct1}, 
\begin{equation*}
         k_h(\zeta,\pD\Omega_\lambda(f))\ge \frac{(\la-1)(1-|\zeta|^2)}{|\zeta|+|f(\zeta)|}.
     \end{equation*}
Since $1-|f(\zeta)|^2=\la(1-|\zeta|^2)$, $\frac{(\la-1)(1-|\zeta|^2)}{|\zeta|+|f(\zeta)|}=|\zeta|-|f(\zeta)|$ and thus \eqref{kh3} holds.
   Note that $|\zeta|-|f(\zeta)|>0$.

The inequality in \eqref{kh3} becomes an equality for some $\zeta\in\pD\Oml(f)$ only if $f=e^{i\alpha}\varphi_a\in{\rm Aut}(\D)$, where $\alpha\in\mathbb{R}$, $a\in\D$, $\la<\frac{1+|a|}{1-|a|}$ (see Corollary \ref{ct1}). In particular, $a,\zeta\in\D\setminus\{0\}$. In this case, $|\overline{f'(\zeta)}f(\zeta)-\lambda\zeta| = \la(|f(\zeta)|+|\zeta|)\Leftrightarrow \left|\frac{a-\zeta}{1-a\ov{\zeta}}+\zeta\right|=\left|\frac{a-\zeta}{1-a\ov{\zeta}}\right|+|\zeta|\Leftrightarrow \ov{\zeta}(a-\zeta)(1-\ov{a}\zeta)\ge0\Leftrightarrow \ov{a}\zeta\in\mathbb{R}\mbox{ and }a\ov{\zeta}(1+|\zeta|^2)\ge|\zeta|^2+|a|^2|\zeta|^2\Leftrightarrow a\ov{\zeta}\in(0,1)\mbox{ and }(1-|a||\zeta|)(|a|-|\zeta|)\ge0\Leftrightarrow a\ov{\zeta}\in(0,1)\mbox{ and }|a|\ge|\zeta|$. Note that $\ov{a}\zeta\in(0,1)\Leftrightarrow \zeta$ is the midpoint of the circular arc $\pD\Oml(f)$. So, by Example \ref{ex1}, $|\zeta|=\frac{1}{|a|}-\sqrt{\frac{\frac{1}{|a|^2}-1}{\la}}\le|a|\Leftrightarrow \la\le\frac{1}{1-|a|^2}$. Hence, the equality in \eqref{kh3} holds for at most one point.
\end{proof}

\begin{corollary}
    \label{c4}
     Let $\la\ge1$ and $f:\D\to\D$ be holomorphic with $f(0)\neq0$, if $\la=1$. Then, for every $\zeta\in\pD\Omega_\lambda(f)$,
     \begin{equation}
     \label{kelb}
k_e(\zeta,\pD\Omega_\lambda(f))\ge -\frac{|\zeta|+|f(\zeta)|}{1-|\zeta|^2}.
     \end{equation}
Moreover, equality holds for some $\zeta\in\pD\Oml(f)$ if and only if $f\in{\rm Aut}(\D)$ with $\la\le\frac{1}{1-|f(0)|^2}$. Furthermore, equality holds for at most one point.
\end{corollary}

\begin{proof}
Since $|\zeta|-|f(\zeta)|=\frac{(\la-1)(1-|\zeta|^2)}{|\zeta|+|f(\zeta)|}$, for every $\zeta\in\pD\Omega_\lambda(f)$, we deduce from Corollaries \ref{cr3} and \ref{cr2} that
$$\frac{1}{1-|\zeta|^2}k_h(\zeta,\pD\Oml(f))\ge \frac{\la-1}{|\zeta|+|f(\zeta)|}.$$
Using \eqref{kh} and $\la(1-|\zeta|^2)=1-|f(\zeta)|^2$,
$$k_e(\zeta,\pD\Oml(f))\ge \frac{\la-1}{|\zeta|+|f(\zeta)|}-\frac{2|\zeta|}{1-|\zeta|^2}=-\frac{(|\zeta|+|f(\zeta)|)^2}{(|\zeta|+|f(\zeta)|)(1-|\zeta|^2)}.$$

  If equality holds for some $\zeta\in\pD\Oml(f)$ in \eqref{kelb}, then $f=e^{i\alpha}\varphi_a\in{\rm Aut}(\D)$, where $\alpha\in\mathbb{R}$, $a\in\D\setminus\{0\}$, $\la\le\frac{1}{1-|a|^2}$ (see Corollaries \ref{cr3} and \ref{cr2}). Moreover, by \eqref{kh}, in the case of equality, ${\rm Re}(\ov{\zeta}n)=-|\zeta|$, which is equivalent with $\zeta$ being the midpoint of the circular arc $\pD\Oml(f)$. So, 
    $k_e(\zeta,\pD\Oml(f))=-\frac{1}{\frac{1}{|a|}-|\zeta|}$ and this is equal to the lower bound if and only if $|f(\zeta)|=\frac{|a|-|\zeta|}{1-|\ov{a}\zeta|}$. This equality holds, since $\ov{a}\zeta\in(0,1)$. So, the equality in \eqref{kelb} holds for at most one point.
\end{proof}

\section{Bounds for hyperbolic curvature of Jordan level curves}
\label{sec5}

In this section, we consider a special case when $\p\Om(f)$ is a Jordan curve that lies in  $\D$, namely, we estimate the hyperbolic curvature of $\p\Om(rf)$, when $r\in(0,1)$. In the following, we denote by $\dt$ the unit circle $\p\D$.

\begin{remark}
    Let $f:\D\to\ov{\D}$ be holomorphic such that $f(0)\neq0$. Then, for every $r\in(0,1)$, $\p\Om(rf)$ is a starlike smooth Jordan curve in $\D$. Indeed, since $\Om(rf)=\{z\in\D:|z|<r|f(z)|\}\subset\{z\in\D:|z|<r\}$, the above holds in view of \cite[Theorem 2.2]{MejPom}.
\end{remark}

\begin{remark}
\label{krrr}
     If we consider an arc of the circle  $\gamma(t)=c+re^{it}$, where  $c\in \C$ and $r>0$ are such that $\{\gamma\}\subset \D$ (i.e., $r-1<|c|<r+1$), then $k_h(z,\gamma)=\cfrac{1-|c|^2+r^2}{r},z\in\{\gamma\}$. Indeed, using  \eqref{kht} and some simple computations, we get the formula. In particular, the hyperbolic curvature of $r\dt$ oriented anticlockwise is $r+\frac{1}{r}$. Hence, the hyperbolic curvature of any hyperbolic circle in $\D$, of hyperbolic radius $r_h>0$, oriented anticlockwise, is $2\coth(2r_h)$ (see \cite[Example 1]{MejMin}).  The circles are the only closed regular $C^2$ curves in $\D$ that have constant hyperbolic curvature (see \cite[Section 2.4]{FerGra}; cf. \cite[Appendix, p. 315]{doC}).
\end{remark}

In the next theorem, which restates Theorem \ref{t03}, we exploit the hyperbolic convexity of the sublevel sets, by using the sharp inequality of Ma and Minda \cite[Theorem 5]{MaMin} and the fixed point function of Mej\'ia and Pommerenke \cite{MejPom}, to get sharp bounds for the hyperbolic curvature of their boundaries. 

\begin{theorem}
\label{tkr}
    Let $f:\D\to\ov{\D}$ be holomorphic such that $f(0)\neq0$. Then the following sharp inequalities hold, for all $r\in(0,1)$ and $\zeta\in\p\Om(rf)$,
    \begin{equation}
    \label{khr}
    \frac{1+r}{1-r}C_{rf,\zeta}-\frac{2(1-r^2)}{r}\frac{1}{C_{rf,\zeta}} \ge k_h(\zeta,\p\Om(rf))\ge\frac{1-r}{1+r}C_{rf,\zeta}+\frac{2(1-r^2)}{r}\frac{1}{C_{rf,\zeta}},
    \end{equation}
    where $C_{rf,\zeta}=|\ov{(rf)'(\zeta)}(rf)(\zeta)-\zeta|\frac{1-|\zeta|^2}{|\zeta|^2}$.
\end{theorem}

\begin{proof} In view of \cite[Section 3]{MejPom}, there exists a conformal map $\psi:\D\to\Om(f)$ such that $\psi(0)=0$, $\psi'(0)=f(0)$ and, for every $w\in\D$, $\psi(w)$ is the unique fixed point of $wf$, i.e., $wf(\psi(w))=\psi(w)$. Hence, $\Om(rf)=\psi(r\D)$. Let $\zeta=\psi(w)\in\p\Om(rf)=\psi(r\dt)$, $|w|=r$. Then $k_h(\zeta,\Om(rf))=k_h(\psi(w),\psi(r\dt))$. Since $r|f(\zeta)|=|\zeta|$, one can easily prove, using \cite[(3.3)]{MejPom}, that
\begin{equation}
\label{eqw}
C_{rf,\zeta}=\frac{1-|\psi(w)|^2}{r|\psi'(w)|}.
\end{equation}

Let 
\begin{equation*}
    p(z)=1+\frac{z\psi''(z)}{\psi'(z)}+\frac{2z\psi'(z)\ov{\psi(z)}}{1-|\psi(z)|^2},\quad z\in\D.
\end{equation*}

By the proof of \cite[Theorem 3, p. 88]{MaMin} (cf. \cite[(6.3)]{Kou1}) and \eqref{eqw},
\begin{equation}
\label{kCf}
{\rm Re}\,p(w)=\frac{k_h(\zeta,\pD\Om(rf))}{C_{rf,\zeta}}
\end{equation}
and (see Remark \ref{Dhj})
\begin{equation}
\label{ppp}
p(w)=\frac{w}{1-|w|^2}\frac{D_{h2}\psi(w)}{D_{h1}\psi(w)}+\frac{1+|w|^2}{1-|w|^2}\;\;\mbox{ and }\;\; C_{rf,\zeta}=\frac{1-r^2}{r|D_{h1}\psi(w)|}.
\end{equation}

Since $\Om(f)$ is hyperbolically convex, $\psi$ is hyperbolically convex, and thus, by \cite[Theorem 5]{MaMin} (see also \cite[p. 29]{Kou1}),
\begin{equation}
\label{kCfineq}
    \left|p(w)-\frac{1+|w|^2}{1-|w|^2}\right|\le \frac{2|w|}{1-|w|^2}\left(1-\left(\frac{(1-|w|^2)|\psi'(w)|}{1-|\psi(w)|^2}\right)^2\right).
\end{equation}
Hence, using \eqref{eqw} and \eqref{kCfineq},
\begin{equation}
\label{kCfineq2}
    \left|{\rm Re}\,p(w)-\frac{1+r^2}{1-r^2}\right|\le \frac{2r}{1-r^2}-\frac{2(1-r^2)}{rC_{rf,\zeta}^2}.
\end{equation}
Hence, by \eqref{kCf}, \eqref{khr} holds.

The equalities hold in \eqref{khr}, if $f\equiv \sigma\in\dt$ (see Remark \ref{krrr}).
\end{proof}

In the following, we point out upper and lower sharp bounds for the hyperbolic curvature of $\p\Om(rf)$, depending only on $r$ and $|f(0)|$, using the function 
\begin{align}
    \label{ka}
    \begin{split}
    k_{\alpha}(z)&=\frac{2\alpha z}{1-z+\sqrt{(1-z)^2+4\alpha^2 z}}\\
    &=\alpha z+\alpha(1-\alpha^2)z^2+\ldots, \quad z\in\D,
    \end{split}
    \end{align}
where $\alpha\in(0,1]$. $k_{\alpha}$ is a conformal map of $\D$ onto
$$\Om=\left\{z\in\D:\left|z+\frac{1}{\alpha}\right|>\sqrt{\frac{1}{\alpha^2}-1}\right\}$$ with $k_{\alpha}^{-1}(z)=g_\alpha(z)$, $z\in\Om$, where $g_\alpha(z)=\frac{z(1+\alpha z)}{\alpha+z},z\in\D$ (see \cite[Example 1]{MaMin}). 

\begin{corollary}
\label{khra}
    Let $f:\D\to\ov{\D}$ be holomorphic such that $f(0)\neq0$. Then, for all $r\in(0,1)$ and $\zeta\in\p\Om(rf)$,
$$ \frac{(1+r)^2}{rD_{h1}k_{\alpha}(-r)}-2D_{h1}k_\alpha(-r)\ge k_h(\zeta,\p\Om(rf))\ge\frac{(1-r)^2}{rD_{h1}k_{\alpha}(r)}+2D_{h1}k_\alpha(-r),$$
where $\alpha=|f(0)|$, $D_{h1}k_{\alpha}(-r)=\frac{\alpha(1-r)}{\sqrt{(1+r)^2-4\alpha^2r}}$, $D_{h1}k_{\alpha}(r)=\frac{\alpha(1+r)}{\sqrt{(1-r)^2+4\alpha^2r}}$. The inequalities are sharp.
\end{corollary}

\begin{proof}
Combining Theorem \ref{tkr} with \cite[Corollary, p. 92]{MaMin} (applied to $C_{rf,\zeta}$ in \eqref{ppp}), the result follows. The equalities hold, if $f\equiv \sigma\in\dt$.
\end{proof}

\section{Radius of convexity for Jordan level curves}
\label{sec6}

For $f:\D\to\ov{\D}$ holomorphic with $f(0)\neq0$, let 
$$\omega_f=\sup\{r\in(0,1):\Om(rf)\mbox{ is convex}\}.$$
Also, let 
$$\omega=\inf\{\omega_f:f:\D\to\ov{\D}\mbox{ is holomorphic with }f(0)\neq0\}.$$

In \cite[Section 5]{MaMin}, Ma and Minda found that the radius of hyperbolic convexity for univalent self-maps of $\D$ is $2-\sqrt{3}$ (which is also the radius of Euclidean convexity for univalent functions from $\D$ to $\C$, see \cite[Theorem 2.2.22]{GrKo}). For other results regarding radii of hyperbolic convexity, see \cite{MaMin2}. In this section, we exploit again the hyperbolic convexity of the sublevel sets to find the radius $\omega$ of (Euclidean) convexity.

\begin{remark}
i) Using the fixed point function given by \cite{MejPom}, we have that, if $\Om(\rho f)$ is convex for some $\rho\in(0,1]$, then $\Om(rf)$ is convex for every $r\in(0,\rho]$ (see \cite[Lemma 6.3.7]{GrKo}).

ii) Let $\Omega\subset\D$ be a domain containing $0$. By a simple geometric argument, we note that, if $\Omega$ is convex, then $\Omega$ is hyperbolically convex. In view of Example \ref{ex1}, the converse does not hold.
\end{remark}

In the next theorem, which restates Theorem \ref{t04}, we find $\omega$, by using Theorem \ref{tkr} and an inequality due to Ma and Minda \cite{MaMin3}. 

\begin{theorem}
\label{r_conv}
   $\omega=\frac{1}{\sqrt 2}$.
\end{theorem}

\begin{proof}
First, we prove $\omega\ge\frac{1}{\sqrt 2}$. Let $f:\D\to\ov{\D}$ be holomorphic with $f(0)\neq0$. Also, let $r\in(0,1)$ and $\zeta\in\p\Om(rf)$.
From \eqref{kh} and the lower bound in Theorem \ref{tkr}, we have
\begin{align*}
\begin{split}
k_e(\zeta,\p\Om(rf))&\ge \frac{k_h(\zeta,\p\Om(rf))-2|\zeta|}{1-|\zeta|^2}\\
&\ge \frac{1}{1-|\zeta|^2}\left(\frac{1-r}{1+r}C_{rf,\zeta}+\frac{2(1-r^2)}{r}\frac{1}{C_{rf,\zeta}}-2|\zeta|\right)\\
&=\frac{(1-r)^2+2r(|D_{h1}\psi(w)|^2-|D_{h1}\psi(w)\cdot\psi(w)|)}{r(1-|\zeta|^2)|D_{h1}\psi(w)|},
\end{split}
\end{align*}
where, at the end, we use \eqref{ppp}, keeping the notations from the proof of Theorem \ref{tkr}. From \cite[p. 283]{MaMin3}, we have the following inequality
$$(|w|-|\psi(w)|^2)|D_{h1}\psi(w)|\ge(1-|w|)|\psi(w)|,$$
which is equivalent with
$$r(|D_{h1}\psi(w)|^2-|D_{h1}\psi(w)\cdot\psi(w)|)\ge x^2-(2r-1)x,$$
where $x=|D_{h1}\psi(w)\cdot\psi(w)|$. Since $x^2-(2r-1)x\ge-\left(r-\frac{1}{2}\right)^2$,  we deduce that 
$k_e(\zeta,\p\Om(rf))\ge0,$ for all $\zeta\in\p\Om(rf)$,
if $r\in(0,\frac{1}{\sqrt 2}]$. So, $\omega_f\ge\frac{1}{\sqrt 2}$.

Next, we shall prove that there exists $f:\D\to\D$ holomorphic with $f(0)\neq0$ such that $\omega_{f}<r$, for every $r\in(\frac{1}{\sqrt 2},1)$. Let $\alpha\in(0,1]$ and $k_{\alpha}$ be given by \eqref{ka}. Let $f_\alpha=-\varphi_{-\alpha}$ and $r\in(0,1)$. In view of \cite{MejPom}, $k_\alpha$ is the fixed point function of $f_{\alpha}$, and thus $\Omega(rf_{\alpha})=k_\alpha(r\D)$. Let $\gamma(t)=k_\alpha(re^{it}), t\in[0,2\pi)$. Then $\gamma$ is a regular parametrization of $\p\Om(rf_\alpha)$ and, using the formula for the Euclidean curvature $k_e$ (see Section \ref{sec1}),  
$$k_e(\gamma(\pi),\p\Om(rf_\alpha))=\frac{{\rm Im}(\gamma''(\pi)\ov{\gamma'(\pi)})}{|\gamma'(\pi)|^3}=\frac{k_\alpha'(-r)-rk_\alpha''(-r)}{r[k_{\alpha}'(-r)]^2},$$
where we use $k_\alpha'(-r)>0$ ($k_\alpha$ is increasing on $(-1,1)$). 
Let $\alpha_0=\frac{\sqrt{\sqrt{2}+2}}{2}$. Since $k_{\alpha_0}(r)=\frac{1}{2\alpha_0}\left(\sqrt{r^2+\sqrt{2}r+1}-1+r\right)$, we get
\begin{align*}
k_e(\gamma(\pi),\p\Om(rf_{\alpha_0}))&=\frac{\left(r^2-\sqrt{2}r+1\right)^{3/2}-r^3 +\frac{3\sqrt{2}}{2} r^2-\frac{5}{2}r+\frac{1}{\sqrt{2}}}{2\alpha_0r[k_{\alpha_0}'(-r)]^2\left(r^2-\sqrt{2}r+1\right)^{3/2}}\\
&=\frac{\left(t^2+\frac{1}{2}\right)^{3/2}-t^3-t-\frac{1}{2\sqrt{2}}}{2\alpha_0r[k_{\alpha_0}'(-r)]^2\left(r^2-\sqrt{2}r+1\right)^{3/2}}<0,
\end{align*}
for $t=r-\frac{1}{\sqrt2}>0$. 
Hence, for every $r\in\left(\frac{1}{\sqrt 2},1\right)$, $\Om(rf_{\alpha_0})$ is not convex.

\end{proof}

Let $\mathcal{K}_h$ denote the family of hyperbolically convex self-maps of $\D$ that fix the origin. As we have seen, to every $f:\D\to\ov\D$ holomorphic with $f(0)\neq0$ there corresponds a fixed point function $\psi\in\mathcal{K}_h$ with $\psi(\D)=\Om(f)$, given by $\psi^{-1}(z)=\frac{z}{f(z)}$,  $z\in\Om(f)$, in view of \cite[(3.2)]{MejPom}. Note that the converse might not hold in general: if $\psi\in\mathcal{K}_h$, then $f(z)=\frac{z}{\psi^{-1}(z)}, z\in\psi(\D)$, might not extend analytically to $\D$. However, retracing the proof of Theorem \ref{r_conv} yields the same radius of (Euclidean) convexity for $\mathcal{K}_h$.

\begin{theorem}
    $\displaystyle \inf_{\psi\in\mathcal{K}_h}\sup\{r\in(0,1):\psi(r\D)\mbox{ is convex}\}=\frac{1}{\sqrt{2}}.$
\end{theorem}

\begin{remark}
 According to \cite{MejPom0}, if $f\in\mathcal{K}_h$, then $\frac{1}{f'(0)}f\in S^*(1/2)$ (i.e., $\frac{1}{f'(0)}f$ is a starlike function of order $1/2$). In view of \cite[Theorem 1]{MacG}, the radius of convexity for $S^*(1/2)$ is $\sqrt{2\sqrt{3}-3}$ (which is less than $\frac{1}{\sqrt{2}}$).  
 
 Note that, by \cite[Theorem 2]{MacG}, $\frac{1}{\sqrt{2}}$ is the radius of starlikeness for normalized holomorphic functions $f:\D\to\C$ with ${\rm Re}\,\frac{f(z)}{z}>\frac{1}{2}$, $z\in\D$. $\frac{1}{\sqrt 2}$ is also known as the value that splits the sharp bound in the Rotation Theorem (see \cite[Theorem 3.2.6]{GrKo}). We don’t know if this
is a coincidence or if there is a deeper connection here.
\end{remark}

\section{Hyperbolic area and hyperbolic perimeter estimates for sublevel sets}
\label{sec7}

In this section, we give some estimates for the total hyperbolic curvature, the hyperbolic area and the hyperbolic perimeter of $\Oml(f)$; see \cite{Kou1,Kou2} for some estimates involving these quantities for hyperbolically convex functions.

\begin{proposition}
\label{p11}
    Let $f:\D\to\D$ be holomorphic and $\la\ge1$. Then
        \begin{equation*}
        A_h(\Oml(f))\ge\max\left\{\pi(\la-1),\frac{\pi}{2}\left(\frac{1}{\sqrt{1-|f(0)|^2}}-1\right)\right\}
    \end{equation*}
    and, for $\partial _{\mathbb{D}}\Omega _{\lambda }(f)\neq\emptyset$,
    \begin{equation*}
        L_h(\pD\Oml(f))\ge\max\left\{2\pi\sqrt{\la(\la-1)},\frac{\pi |f(0)|}{\sqrt{1-|f(0)|^2}}\right\}.
    \end{equation*}
    
    In particular, $\frac{1}{\la^p}A_h(\Oml(f))\to\infty$ and, for $f$ not a finite Blaschke product,  $\frac{1}{\la^p}L_h(\pD\Oml(f))\to\infty$, as $\la\to\infty$, for every $p<1$.
\end{proposition}

\begin{proof}
 If $f(0)=0$ and $\la=1$, there is nothing to prove. Let $\la>1$. Then $r_\la\D\subseteq\Oml(f)$ with $r_\la=\sqrt{1-\frac{1}{\la}}$, so 
 $$A_h(\Oml(f))\ge \iint_{r_\la\D}\la^2_\D(z){\rm d}A(z)=\frac{\pi r_\la^2}{1-r_\la^2}=\pi(\la-1)$$
 and
  $$L_h(\pD\Oml(f))\ge \int_{r_\la\dt}\la_\D(z){\rm d}z=\frac{2\pi r_\la}{1-r_\la^2}=2\pi\sqrt{\la(\la-1)},$$
  where we use the monotonicity of the hyperbolic perimeter with respect to inclusion for hyperbolically convex sets (easily seen to hold by intersecting with hyperbolic half-planes and using the triangle inequality).
Note that, if $\p\Oml(f)\cap\p\D\neq\emptyset$, then $\pD\Oml(f)$ is a union of Jordan arcs with endpoints on $\p\D$, so $L_h(\pD\Oml(f))=\infty$, while, if $\p\Oml(f)\subset\D$, then $\p\Oml(f)$ has finite hyperbolic length. Also, if $\pD\Oml(f)=\emptyset$, then $L_h(\pD\Oml(f))=0$.
 
 Next, let $f(0)\neq0$. In view of \cite[Section 3]{MejPom}, there exists an univalent function $\psi:\D\to\Om(f)$ such that $\psi(0)=0$ and $\psi'(0)=f(0)$. By \cite[Theorem 2]{MaMin},
 $r\D\subset\Om(f)$, where $r=\frac{|f(0)|}{1+\sqrt{1-|f(0)|^2}}$. Since  $\Om(f)\subseteq\Oml(f)$, the proof is complete, by simple computations as above.
\end{proof}

As a consequence, we get a lower bound for the total hyperbolic curvature, related to Fenchel's Theorem \cite[Section 5.7, Theorem 3]{doC}. For sharp bounds, see Corollary \ref{kh_bnds}.

\begin{corollary}
    Let $f:\D\to\ov{\D}$ be holomorphic such that $f(0)\neq0$ and $r\in(0,1)$. Then
$$k_h(\p\Om(rf))>\frac{2\pi}{\sqrt{1-r^2|f(0)|^2}}.$$
\end{corollary}

\begin{proof}
    We apply Proposition \ref{p11} to $rf$ and $\la=1$ and we use \eqref{GB}. The inequality is strict, in view of \cite[Theorem 2]{MaMin} (which was used in the proof of Proposition \ref{p11}) and  the fact that $rf$ cannot be an automorphism.
\end{proof}

\begin{proposition}
\label{Lh_low}
    Let $f:\D\to\ov{\D}$ be holomorphic such that $f(0)\neq0$ and $r\in(0,1)$. Then the following sharp inequality holds
    $$L_h(\p\Om(rf))\ge\frac{2\pi r|f(0)|}{(1+r)\sqrt{(1+r)^2-4r|f(0)|^2}}.$$
\end{proposition}

\begin{proof}
    Using the fixed point function given by \cite{MejPom}, the inequality follows easily from \cite[Corollary, p. 92]{MaMin} and the definition of $L_h$. The equality holds for $f\equiv\sigma\in\dt$.
\end{proof}

In the following theorem, we point out sharp upper bounds for the hyperbolic area and the hyperbolic perimeter of $\Om(rf)$, $r\in(0,1)$, as immediate consequences of the sharp estimates obtained by Kourou \cite{Kou2}. We mention that \eqref{ah2} contains the upper bound of Theorem \ref{t05}.

\begin{theorem}
\label{tAh1}
    Let $f:\D\to\ov{\D}$ be holomorphic such that $f(0)\neq0$ and $r\in(0,1)$. Then
    \begin{equation}
    \label{ah2}
        A_h(\Om(rf))\le \frac{\pi r^2|f(0)|^2}{1-r^2},\quad L_h(\p\Om(rf))\le \frac{2\pi r|f(0)|}{1-r^2}.
    \end{equation}
    and
    \begin{equation}
    \label{ah22}
        L_h^2(\p\Om(rf))\le \frac{4\pi }{1-r^2}A_h(\Om(rf)).
    \end{equation}
    Moreover, for each inequality, the equality holds if and only if $f\equiv\sigma\in\dt$.
\end{theorem}

\begin{proof}
    Let $\psi:\D\to\Om(f)$ be the unique conformal map with $\psi(0)=0$ and $\psi'(0)=f(0)$. Then $\Om(rf)=\psi(r\D)$. The desired inequalities follow from \cite[Corollaries 1.2, 1.3, 1.4]{Kou2}. For each inequality, the equality holds  if and only if $\psi(z)=\sigma z, z\in\D$, for some $\sigma\in\dt$. Since $f(z)=\frac{z}{\psi^{-1}(z)}, z\in\Om(f)$ (by \cite[Section 3]{MejPom}), $f\equiv \sigma \in\dt$. 
\end{proof}

\begin{remark}
 The sharp upper bound for the Euclidean area of $\Om(f)$ was obtained by Mej\'ia and Pommerenke \cite[Theorem 3.1]{MejPom}. An upper bound for the Euclidean perimeter of $\Om(f)$ is $\pi^2$, in view of the general result of Flinn \cite[Theorem 3]{Fl}, for any hyperbolically convex subset of $\D$. Since $\Om(f)$ is hyperbolically convex, using the fixed point function given by \cite{MejPom} and the inequality \cite[Theorem 6.1]{MaMin3}, one can find some corresponding lower bounds.
\end{remark}

The following consequence of the isoperimetric-type inequality \eqref{ah22} combined with Proposition \ref{Lh_low} gives a sharp lower bound  for the hyperbolic area, which is also stated in Theorem \ref{t05}.

\begin{corollary}
    \label{Ah_low}
    Let $f:\D\to\ov{\D}$ be holomorphic such that $f(0)\neq0$ and $r\in(0,1)$. Then 
    $$A_h(\Om(rf))\ge\frac{\pi(1-r)r^2|f(0)|^2}{(1+r)((1+r)^2-4r|f(0)|^2)},$$
    with equality if and only if  $f\equiv\sigma\in\dt$.
\end{corollary}

The following consequence of Theorem \ref{tAh1} and Corollary \ref{Ah_low} combined with \eqref{GB} gives sharp bounds for the hyperbolic total curvature. 

\begin{corollary}
\label{kh_bnds}
    Let $f:\D\to\ov{\D}$ be holomorphic such that $f(0)\neq0$ and $r\in(0,1)$. Then 
    $$2\pi+\frac{4\pi(1-r)r^2|f(0)|^2}{(1+r)((1+r)^2-4r|f(0)|^2)}\le k_h(\p\Om(rf))\le 2\pi+\frac{4\pi r^2|f(0)|^2}{1-r^2}.$$
    Moreover, for each inequality, the equality holds if and only if  $f\equiv\sigma\in\dt$.
\end{corollary}

\begin{remark}
The proof of the upper bound of Corollary \ref{kh_bnds} implies \cite[Corollary 1.1]{Kou1}. Indeed, if  $\psi:\D\to\D$ is hyperbolically convex, then, by \cite[Corollary 1.4]{Kou2} and \eqref{GB}, 
$$k_h(\psi(r\dt))\le 2\pi + \frac{|\psi'(0)|^2}{(1-|\psi(0)|^2)^2}\frac{4\pi r^2}{1-r^2}.$$
This upper bound is less or equal than $\frac{2\pi(1+r^2)}{1-r^2}$, by the Schwarz-Pick Lemma.
\end{remark}

The following sharp hyperbolic isoperimetric inequality follows from   \eqref{isop} and Corollary \ref{Ah_low}.

\begin{corollary}
    Let $f:\D\to\ov{\D}$ be holomorphic such that $f(0)\neq0$ and $r\in(0,1)$. Then 
    \begin{equation}
\label{AhLh}
L_h^2(\p\Om(rf)) \ge 4\pi \left(1+\frac{(1-r)r^2|f(0)|^2}{(1+r)((1+r)^2-4r|f(0)|^2)}\right) A_h(\Om(rf)),
\end{equation}
    with equality if and only if  $\Om(rf)=r\dt$.
\end{corollary}

\begin{proof}
\eqref{isop} and Corollary \ref{Ah_low} imply
\begin{align*}
L_h^2(\p\Om(rf)) &\ge 4\pi \left(1+\frac{A_h(\Om(rf))}{\pi}\right) A_h(\Om(rf)) \\
&\ge 4\pi \left(1+\frac{(1-r)r^2|f(0)|^2}{(1+r)((1+r)^2-4r|f(0)|^2)}\right) A_h(\Om(rf)). 
\end{align*}
The equality in the second inequality holds if and only if $f\equiv\sigma\in\dt$, which is equivalent to $\p\Om(rf)=r\dt$. Finally, we note that the equality in the first inequality holds if and only if $\p\Om(rf)$ is a circle.
\end{proof}

\section*{Acknowledgements}
The authors thank the anonymous referee for a careful reading of the manuscript and helpful comments.

\end{document}